\documentclass[12pt,twoside]{amsart}
\usepackage{amsmath}
\usepackage{amsthm}
\usepackage{amsfonts}
\usepackage{amssymb}
\usepackage{latexsym}
\usepackage{mathrsfs}
\usepackage{amsmath}
\usepackage{amsthm}
\usepackage{amsfonts}
\usepackage{amssymb}
\usepackage{latexsym}
\usepackage{geometry}
\usepackage{dsfont}
\usepackage[dvips]{graphicx}
\usepackage{color}
\usepackage[all]{xy}

\date{}
\allowdisplaybreaks[4] \footskip=15pt
\renewcommand{\uppercasenonmath}[1]{}

\usepackage{graphicx,amssymb}
\usepackage[all]{xy}
\usepackage{amsmath}

\allowdisplaybreaks
\usepackage{amsthm}
\usepackage{color}

\theoremstyle{plain}
\newtheorem{theorem}{Theorem}[section]

\newtheorem{lemma}[theorem]{Lemma}
\newtheorem{corollary}[theorem]{Corollary}
\newtheorem{example}[theorem]{Example}
\newtheorem*{open question}{Open Question}

\theoremstyle{definition}
\newtheorem*{acknowledgement}{Acknowledgement}

\theoremstyle{remark}

\newcommand{\Tor}{\mbox{\rm Tor}}

\def\p{\frak p}

\def\GV{{\rm GV}}
\def\tor{{\rm tor_{\rm GV}}}
\def\Hom{{\rm Hom}}
\def\Ext{{\rm Ext}}
\def\Tor{{\rm Tor}}

\def\Ker{{\rm Ker}}

\def\GV{{\rm GV}}

\begin{document}
\begin{center}
{\large  \bf A note on a Cohen-type theorem for $w$-Artinian modules}

\vspace{0.5cm}   Xiaolei Zhang$^{a}$

{\footnotesize  a.\ School of Mathematics and Statistics, Shandong University of Technology, Zibo 255049, China\\

E-mail: zxlrghj@126.com\\}
\end{center}

\bigskip
\centerline { \bf  Abstract}
\bigskip
\leftskip10truemm \rightskip10truemm \noindent

In this note, we prove that a $w$-module $M$ is $w$-Artinian if and only if  it is $w$-cofinitely generated and  for every prime $w$-ideal $\p$ of $R$ with $(0:_RM)\subseteq \p$, there exists a $w$-submodule $N^\p$ of $M$  such that $(M/N^\p)_w$ is $w$-cofinitely generated and $(M[\p])_w\subseteq N^\p\subseteq (0:_M\p)$, where $M[\p]=\bigcap\limits_{s\in R \setminus \p}s(0:_M\p).$ Besides, we show that every principal ideal is not a $w$-ideal for some Noetherian rings.
\vbox to 0.3cm{}



\leftskip0truemm \rightskip0truemm
\bigskip

\section{Introduction}
Throughout this article, all rings are commutative rings with identity and all modules are unitary. Let $R$ be a ring, $I$ an ideal of $R$ and $M$ an $R$-module. We denote by $(0:_RM):=\{r\in R\mid rM=0\}$ and $(0:_MI):=\{m\in M\mid Im=0\}$.
The well-known Cohen's Theorem states that a ring $R$ is a Noetherian ring  if and only if every prime ideal  $\p$ of $R$ is finitely generated (see \cite[Theorem 2]{c50}). In 1994, Smith \cite{s94} extended  Cohen's Theorem from rings to  modules, that is, a finitely generated $R$-module $M$ is Noetherian if and only if the submodules $\p M$ of $M$ are finitely generated for every prime ideal $\p$ of $R$, if and only if $M(\p)$ is finitely generated  for each prime ideal $\p$ of $R$ with $(0:_RM)\subseteq\p$, where $M(\p)=\{x\in M \mid sx\in \p M $ for some $s\in R \setminus \p \}$.  In 2021, Parkash and  Kour \cite{pk21} generalized the Smith's result on Noetherian modules and obtained that a finitely generated $R$-module $M$ is Noetherian if and only if for every prime ideal $\p$ of $R$ with $(0:_RM)\subseteq \p$, there exists a finitely generated  submodule $N_\p$ of $M$ such that $\p M\subseteq N_\p\subseteq M(\p)$. Recently, the  author et al. \cite{zkq23} gave a $w$-analogue of Parkash and  Kour's result which states that a  $\GV$-torsion-free $w$-finite type $R$-module $M$ is  $w$-Noetherian  if and only if for every prime $w$-ideal $\p$ of $R$ with $(0:_RM)\subseteq \p$, there exists a $w$-finite type submodule $N^{\p}$ of $M$ such that $\p M\subseteq N^{\p}\subseteq M(\p)$.

Recall that  an $R$-module $M$ is said to be {\em Artinian} if it satisfies the minimal condition for submodules, or equivalently, the descending chain condition on submodules. And $M$ is said to be {\em cofinitely generated} (which is also called {\em finitely embedded} in some other papers) if  for any set $\{M_i | i\in \Omega\}$  of submodules of $M$ satisfying $\bigcap\limits_{i\in \Omega}M_i=0$, there exists a finite subset $\Omega_0\subseteq \Omega$  such that  $\bigcap\limits_{i\in \Omega_0}M_i=0$.  A family $\{M_i\}_{i\in \Lambda}$ of submodules of $M$ is called an {\em inverse system} if for any finite number of $i_1,i_2,\dots,i_k$ of an index set $\Lambda$, there is an element $i\in \Lambda$ such that $M_i\subseteq \bigcap\limits_{j=1}^k M_{i_j}$. By \cite[Proposition 3.19]{sv72}, $M$ is cofinitely generated if and only if every inverse system of nonzero submodules of $M$ is bounded below by a nonzero submodule of $M$.

It is well known that a Noetherian module is exactly a module of which all submodules are finitely generated. Dually, an $R$-module $M$ is  Artinian if and only if every factor module of $M$ is cofinitely generated  (see  \cite[Theorem 3.21]{sv72}). In 2006, Nishitani \cite{N06} obtained a Cohen-type theorem for Artinian modules: A finitely embedded module $M$ is Artinian if and only if $M/(0:_M\p)$ is cofinitely generated  for every prime ideal $\p$ of $R$. Recently, the author et al.  generalized the Nishitani's result and dualized   the Parkash and  Kour's result as follows:
\begin{theorem}\cite[Theorem 2.1]{zkq23-1}
A finitely embedded  $R$-module $M$ is Artinian if and only if for every prime ideal $\p$ of $R$ with $(0:_RM)\subseteq \p$, there exists a submodule $N^\p$ of $M$  such that $M/N^\p$ is finitely  embedded and $M[\p]\subseteq N^\p\subseteq (0:_M\p)$, where $M[\p]=\bigcap\limits_{s\in R \setminus \p}s(0:_M\p).$
\end{theorem}


The main motivation of this note is to give a $w$-analogue of \cite[Theorem 2.1]{zkq23-1}. We recall some notions on the $w$-operation. Let $R$ be a commutative ring and $J$ a finitely generated ideal of $R$. Then $J$ is called a \emph{$\GV$-ideal} if the natural homomorphism $R\rightarrow \Hom_R(J,R)$ is an isomorphism. The set of $\GV$-ideals is denoted by $\GV(R)$. Let $M$ be an $R$-module. Define

\begin{center}
{\rm $\tor(M):=\{x\in M|Jx=0$ for some $J\in \GV(R) \}.$}
\end{center}
An $R$-module $M$ is said to be \emph{$\GV$-torsion} (resp. \emph{$\GV$-torsion-free}) if $\tor(M)=M$ (resp. $\tor(M)=0$). A $\GV$-torsion-free module $M$ is called a \emph{$w$-module} if $\Ext_R^1(R/J,M)=0$ for any $J\in \GV(R)$, and the \emph{$w$-envelope} of $M$ is given by
\begin{center}
{\rm $M_w:=\{x\in E(M)|Jx\subseteq M$ for some $J\in \GV(R) \},$}
\end{center}
where $E(M)$ is the injective envelope of $M$. Therefore, a $\GV$-torsion-free module $M$ is a $w$-module if and only if $M_w=M$. The class of $w$-modules is closed under direct limits and inverse limits (see {\cite[Theorem 7, Theorem 11]{DF16}}). Let $0\rightarrow M\rightarrow N\rightarrow L\rightarrow 0$ be a short exact sequence. It is easy to verify that if  $N$ is a $\GV$-torsion-free $R$-module and $M$ is a $w$-module, then $L$ is a $\GV$-torsion-free $R$-module.

Recall from \cite[Definition 6.9.1]{fk16} that a $w$-module $M$ is said to be  {\em $w$-Artinian} if $M$ satisfies the descending chain condition on $w$-submodules
of $M$. Clearly every Artinian module is $w$-Artinian, but the converse does not hold (see \cite[Example
6.9.7]{fk16}).  By \cite[Theorem 6.9.2]{fk16}, a $w$-module $M$ is $w$-Artinian if and only if it has the minimal condition on $w$-submodules of $M$, if and only if for any set $\{M_i | i\in \Omega\}$ of $w$-submodules of $M$, there is a finite subset $\Omega_0\subseteq \Omega$ such that $\bigcap\limits_{i\in \Omega}M_i = \bigcap\limits_{i\in \Omega_0}M_i$. To extend the notion of $w$-Artinian modules, Zhou, Kim and Hu
\cite[Definition 2.1]{zkh21} called a $w$-module $M$ {\em $w$-cofinitely generated} if for any set $\{M_i | i\in \Omega\}$  of $w$-submodules of $M$ satisfying $\bigcap\limits_{i\in \Omega}M_i=0$, there exists a finite subset $\Omega_0\subseteq \Omega$  such that  $\bigcap\limits_{i\in \Omega_0}M_i=0$. They showed that a $w$-module $M$ is $w$-cofinitely generated if and only if it is an essential extension of a $w$-Artinian module,  if and only if every inverse system of nonzero $w$-submodules of $M$ is bounded below by a nonzero $w$-submodule of $M$ (see \cite[Theorem 2.4, Proposition 2.11]{zkh21}). And finally, they obtained a Cohen-type Theorem for $w$-Artinian modules, which can be seen as a $w$-analogue of Nishitani's result in \cite{N06}:
\begin{theorem}\cite[Theorem 4.10]{zkh21} Let $R$ be a ring. A $w$-module $M$ over $R$ is $w$-Artinian if and only if $M$ is $w$-cofinitely generated and  $(M/(0:_M\p))_w$ is $w$-cofinitely generated  for every prime $w$-ideal $\p$ of $R$.
\end{theorem}

In this note, we gave a new Cohen-type Theorem for $w$-Artinian modules which generalizes \cite[Theorem 4.10]{zkh21} and \cite[Theorem 2.1]{zkq23-1}. Actually, we obtain the following main result of this note.\\
{\bf Theorem 2.5} Let $R$ be a ring and $M$ a $w$-module. Then the following statements are equivalent:
 \begin{enumerate}
\item $M$ is $w$-Artinian;
 \item  $M$ is $w$-cofinitely generated and  for every prime $w$-ideal $\p$ of $R$ with $(0:_RM)\subseteq \p$, there exists a $w$-submodule $N^\p$ of $M$  such that $(M/N^\p)_w$ is $w$-cofinitely generated and $(M [\p])_w\subseteq N^\p\subseteq (0:_M\p)$, where $M[\p]=\bigcap\limits_{s\in R \setminus \p}s(0:_M\p).$
 \end{enumerate}

We recall some notions on  star operations. Let $R$ be a ring, $Q$ its total quotient ring and $\overline{\mathscr{F}}(R)$ the set of all submodules $A$ of $Q$. Recall from \cite{ZKWC} (also see \cite{N15}) that a set map $\star:\overline{\mathscr{F}}(R)\rightarrow \overline{\mathscr{F}}(R)$ is said to be a  \emph{star operation} provided
it satisfies the following properties: for all $A, B\in \overline{\mathscr{F}}(R)$, we have
 \begin{enumerate}
\item  (Extension) $A\subseteq A_{\star}$;
\item (Order-preservation) If $A\subseteq B$ then $A_{\star}\subseteq B_{\star}$;
\item (Idempotence) $A_{\star}= (A_{\star})_{\star}$;
\item (Sub-multiplication) $A_{\star}B_{\star}\subseteq(AB)_{\star}$;
\item (unital) $ R_{\star}=R$.
 \end{enumerate}
Moreover,  if  for any $a\in R$ and $A\in \overline{\mathscr{F}}(R)$, we have

 ~ (6) (Principal) $(aR)_{\star}=aR$, $(aA)_{\star}=aA_{\star}$.\\
then we call $\star$ is  a \emph{principally star operation}. Note that $(6)$ essentially comes from  Gilmer \cite[Section 32]{G72} which was
first proposed by Krull \cite{K35} to study the more generalized Kronecker function ring and
renamed it to its current name by Gilmer in \cite{G72}.  Certainly, the $w$-operation is a  star operation on a ring, and a principally  star operation on an integral domain. It is an interesting question  whether the $w$-operation is a principally star operation on any ring. In this note,  we show that a principal ideal need not be a $w$-ideal even for Noetherian rings, so the $w$-operation is, generally, a star operation rather than a principally  star operation (see Example  \ref{prin-not-w}).

\section{Results}

We begin with following easy result.
\begin{lemma}\label{ann-w} Let $R$ be a ring, $I$ an ideal of $R$ and $M$ a $w$-module over $R$.  Then $(0:_MI)$ is a $w$-submodule of $M$.
\end{lemma}
\begin{proof} First assume $I=Rr$ is a principal ideal of $R$. Suppose $Jm\subseteq (0:_Mr)$, where $J\in \GV(R)$ and $m\in M$. Then $Jmr=0$, and so $mr=0$ since $M$ is $\GV$-torsion-free. Hence $m\in (0:_Mr)$, and so $(0:_Mr)$ is a $w$-submodule of $M$ by \cite[Theorem 6.1.16]{fk16}. Now assume $I$ is an arbitrary ideal of $R$. Then $(0:_MI)=\bigcap\limits_{r\in I}(0:_Mr)$.  So $(0:_MI)$ is a $w$-submodule of $M$ by \cite[Corollary 6.5(1)]{K08}.
\end{proof}

Recall some  notations from \cite[Chapter 18]{G86}  on  $\tau_w$-cofinitely generated modules, where the hereditary torsion theory $\tau_w$ is induced
by the Gabriel topology
$\mathcal{F} := \{I | I$ is an ideal of $R$ with $I_w = R\}$. Let $M$ be an $R$-module and $N$  an $R$-submodule of $M$. Then $N$ is said to be  \emph{$\tau_w$-pure}  in $M$ if $M/N$ is $\GV$-torsion-free.
An $R$-module $M$ is said to be  \emph{$\tau_w$-cofinitely generated} if for any set
$\{M_i | i\in \Omega\}$ of $\tau_w$-pure submodules of $M$ satisfying $\bigcap\limits_{i\in \Omega}M_i = \tor(M)$, there exists a finite
subset $\Omega_0$ of $\Omega$ such that $\bigcap\limits_{i\in \Omega_0}M_i = \tor(M)$.

\begin{lemma}\cite[Lemma 2.9]{zkh21}\label{tw-prop} The following statements hold.
 \begin{enumerate}
\item  A submodule of a $\tau_w$-cofinitely generated R-module is $\tau_w$-cofinitely generated.
\item   If the sequence $0 \rightarrow A\rightarrow  B\rightarrow  C\rightarrow  0$ is exact with $A$ and $C$ $\tau_w$-cofinitely generated, then
B is $\tau_w$-cofinitely generated.
 \end{enumerate}
\end{lemma}

\begin{lemma}\cite[Proposition 2.10]{zkh21}\label{tw-w-cof}
 The following statements are equivalent for a $\GV$-torsion-free $R$-module $M.$
\begin{enumerate}
\item  $M_w$ is $w$-cofinitely generated.
\item  $M_w$ is $\tau_w$-cofinitely generated.
\item  $M$ is $\tau_w$-cofinitely generated.
 \end{enumerate}
\end{lemma}

\begin{corollary}\label{w-cof-prop}
Let $0\rightarrow A\rightarrow B\rightarrow C\rightarrow 0$ be a short exact sequence of $\GV$-torsion-free $R$-modules. Then the following statements hold.
\begin{enumerate}
\item  If $B_w$ is  $w$-cofinitely generated, so is  $A_w$.
\item  If $A_w$ and  $C_w$ are $w$-cofinitely generated, so is $B_w$.
 \end{enumerate}
\end{corollary}
\begin{proof} These follow by Lemma \ref{tw-prop} and Lemma \ref{tw-w-cof}.
\end{proof}

Let $R$ be a ring, $\p$ be a prime ideal of $R$, and $M$ an $R$-module. Following \cite{zkq23-1}, set $$M[\p]=\bigcap\limits_{s\in R \setminus \p}s(0:_M\p).$$ Then $M[\p]$ is a submodule of $M$. We are ready to state and prove the main result of this note.

\begin{theorem}\label{main} Let $R$ be a ring and $M$ a $w$-module. Then the following statements are equivalent:
 \begin{enumerate}
\item $M$ is $w$-Artinian;
 \item  $M$ is $w$-cofinitely generated and  for every prime $w$-ideal $\p$ of $R$ with $(0:_RM)\subseteq \p$, there exists a $w$-submodule $N^\p$ of $M$  such that $(M/N^\p)_w$ is $w$-cofinitely generated and $(M [\p])_w\subseteq N^\p\subseteq (0:_M\p)$.
 \end{enumerate}
\end{theorem}

\begin{proof} $(1)\Rightarrow (2)$ Assume that the $w$-module $M$ is a $w$-Artinian $R$-module.  Let $\p$ be a prime $w$-ideal with $(0:_RM)\subseteq \p$.  Set $N^\p:=(0:_M\p)$. Then $N^\p$ is a $w$-submodule of $M$  by Lemma \ref{ann-w}. It follows by \cite[Theorem 4.10]{zkh21} that $(M/N^\p)_w$ is $w$-cofinitely generated with $M[\p]\subseteq N^\p\subseteq (0:_M\p)$.

$(2)\Rightarrow (1)$ On contrary, suppose that $M$ is not $w$-Artinian. Then there exists  a $w$-submodule $N'$ of $M$ such that $(M/N')_w$ is not $w$-cofinitely generated by \cite[Theorem 4.7]{zkh21}. Consider the set $$\Gamma:=\{N \mid N ~\mbox{is a} ~w\mbox{-submodule of}~M ~\mbox{and}~(M/N)_w ~\mbox{is not $w$-cofinitely generated}\}.$$ Then $\Gamma$ is not empty since $N'\in \Gamma$. Make a partial order on $\Gamma$ by the opposite of inclusion, that is, $N_1\geq N_2$ if and only if $N_1\subseteq N_2$ in $\Gamma$. We will prove the following three claims.

{\bf Claim 1: There exists a maximal element $N\in \Gamma$.}  Let $\{N_i\mid i\in \Lambda\}$ be a total ordered subset of $\Gamma$. Set $N=\bigcap\limits_{i\in \Lambda}N_i$. Then $N$ is a $w$-module by \cite[Theorem 11]{DF16}. We claim that $(M/N)_w$ is not $w$-cofinitely generated. Indeed, since $\{N_i\mid i\in \Lambda\}$ is a total ordered, we have $\{(N_j/N)_w\}_{j\in\Lambda}$ is an  inverse system of submodules of $(M/N)_w$. By \cite[Proposition 2.11]{zkh21}, there are two possibilities: either $(N_j/N)_w=0$ for some $j\in \Lambda$, or $(M/N)_w$ is not $w$-cofinitely generated. In the former case, $N=N_j$ and so $(M/N)_w$ is not $w$-cofinitely generated. Hence both cases imply that $(M/N)_w$ is not $w$-cofinitely generated, and so the totally ordered subset of $\Gamma$ is bounded above by $N$. Consequently, by Zorn's Lemma, $\Gamma$ has a maximal element, which is also denoted by $N$.
Set $$\p=(0:_RN).$$

{\bf Claim 2: $\p$ is a prime $w$-ideal of $R$.} It follows by \cite[Proposition 6.1.20]{fk16} that $\p$ is a  $w$-ideal of $R$. Next we will show  $\p$ is a prime ideal of $R$. Indeed, let $a\not\in \p, b\not\in \p$ be elements in $R$. Then $(0:_Na)\subsetneq N$.  Since $(0:_Na)$ is a $w$-module by Lemma \ref{ann-w}, $(M/(0:_Na))_w$ is $w$-cofinitely generated by the maximality of $N$. So the submodule $((0:_Ma)/(0:_Na))_w$ is also $w$-cofinitely generated by Corollary \ref{w-cof-prop}(1).  Consider the exact sequence $$0\rightarrow (0:_Ma)/(0:_Na)\rightarrow M/N\rightarrow aM/aN\rightarrow 0.$$
Since $(M/N)_w$ is  not $w$-cofinitely generated and $((0:_Ma)/(0:_Na))_w$ is $w$-cofinitely generated, $aM/aN$ is not $\tau_w$-cofinitely generated by Lemma \ref{tw-prop}.

{\bf Subclaim:  $(aM+(aN)_w)/(aN)_w$ is also  not $\tau_w$-cofinitely generated.} We will show it by proof by contradiction.
For any $\{M'_i/aN | i\in \Omega\}$ of $\tau_w$-pure submodules of $aM/aN$ with  $\bigcap\limits_{i\in \Omega}M'_i/aN = \tor(aM/aN)$, where $M'_i\subseteq aM$, then $\{M'_i+(aN)_w /(aN)_w | i\in \Omega\}$ of $\tau_w$-pure submodules of $(aM+(aN)_w)/(aN)_w$ satisfying $\bigcap\limits_{i\in \Omega}(M'_i+(aN)_w)/(aN)_w = \tor((aM+(aN)_w)/(aN)_w)=0$. Indeed, let $x+(aN)_w\in \bigcap\limits_{i\in \Omega}(M'_i+(aN)_w)/(aN)_w$ with $x\in \bigcap\limits_{i\in \Omega}(M'_i+(aN)_w)$. Then there exists $J\in \GV(R)$ such that $Jx\subseteq (aN)_w$. And so $x\in (aN)_w$. Suppose, on the contrary, that $(aM+(aN)_w)/(aN)_w$ is  $\tau_w$-cofinitely generated. Then  there exists a finite subset $\Omega_0\subseteq \Omega$ such that $\bigcap\limits_{i\in \Omega_0}(M'_i+(aN)_w)/(aN)_w=0$, that is, $\bigcap\limits_{i\in \Omega_0}M'_i\subseteq (aN)_w$. We will show  $$\bigcap\limits_{i\in \Omega_0}M'_i/aN = \tor(aM/aN),$$ which  contradicts that $aM/aN$ is not $\tau_w$-cofinitely generated. Indeed, since $\Omega_0$ is finite, there exists $J\in \GV(R)$ such that $J\bigcap\limits_{i\in \Omega_0}M'_i\subseteq aN$. Hence $\bigcap\limits_{i\in \Omega_0}M'_i/aN = \tor(aM/aN)$.

Now we are ready to prove that $\p$ is a prime ideal of $R$. By Lemma \ref{tw-w-cof}, we have $((aM+(aN)_w)/(aN)_w)_w$ is not $w$-cofinitely generated. Thus $(M/(a N)_w)_w$ is  not $w$-cofinitely generated by Corollary \ref{w-cof-prop}(1). So $(aN)_w=N$ by the maximality of $N$. Similarly, we have $(bN)_w=N$. Hence, by \cite[Theorem 6.2.2]{fk16}, $(abN)_w=(a(bN)_w)_w=(aN)_w=N\not=0$ as $M$ is  $w$-cofinitely generated. So $abN\not=0$, and hence $ab\not\in\p$.

{\bf Claim 3: $N\subseteq (M[\p])_w$.} Indeed, suppose that there is  $y\in N$ such that $y\not\in  (M[\p])_w$. Then for any $J\in \GV(R)$, $Jy\not\subseteq M[\p]=\bigcap\limits_{s\in R \setminus \p}s(0:_M\p)$. And so $Jy\not\subseteq (s(0:_M\p))_w$ for some $s\in R \setminus \p$, that is, $y\not\in (s(0:_M\p))_w$ for some $s\in R \setminus \p$.  It follows that $(sN)_w\subseteq (s(0:_M\p))_w\subsetneq N$. And hence $(M/(sN)_w)_w$ is  $w$-cofinitely generated by the maximality of $N$. Since $s\not\in\p$, we have $(0:_Ns)\subsetneq N$.  So $(M/ (0:_Ns))_w$ is also $w$-cofinitely generated by the maximality of $N$. Consider the exact sequence $$0\rightarrow (0:_Ms)/ (0:_Ns)\rightarrow M/N\rightarrow sM/sN\rightarrow 0.$$
Since $(M/(sN)_w)_w$ is  $w$-cofinitely generated, the submodule $((sM+(sN)_w)/(sN)_w)_w$ is also $w$-cofinitely generated by Corollary \ref{w-cof-prop}. By the proof of  {\bf subclaim} in that of {\bf Claim 2}, $sM/sN$ is $\tau_w$-cofinitely generated. Since $(M/ (0:_Ns))_w$ is $w$-cofinitely generated, the submodule $((0:_Ms)/ (0:_Ns))_w$ is also $w$-cofinitely generated. Hence $(M/N)_w$ is $w$-cofinitely generated, which is a contradiction.

Finally, we will show that $M$ is $w$-Artinian. Suppose that the $w$-cofinitely generated  $R$-module $M$ is not $w$-Artinian. Then, by \cite{zkh21}, there is a $w$-ideal $I$ of $R$ such that $(0:_MI)$ is $w$-Artinian and $(M/(0:_MI))_w$ is not $w$-cofinitely generated by \cite[Lemma 7]{N06}. Furthermore there is a $w$-submodule  $N$ of $(0:_MI)$ such that $(M/N)_w$ is not $w$-cofinitely generated and $\p=(0:_RN)$ is a  prime $w$-ideal by {\bf Claim 1} and {\bf Claim 2}. Since $N\subseteq (0:_MI)$, we have  $(0:_M\p)\subseteq (0:_MI)$. Thus $((0:_M\p)/N)_w$ is $w$-Artinian, and hence is $w$-cofinitely generated. Since  $(0:_RM)\subseteq \p$, there is a $w$-submodule $N^\p$ of $M$  such that $(M/N^\p)_w$ is $w$-cofinitely generated with $N\subseteq (M[\p])_w\subseteq N^\p\subseteq (0:_M\p)$ by assumption and {\bf Claim 3}. Then the submodule $(N^{\p}/N)_w$ of $((0:_M\p)/N)_w$ is $w$-cofinitely generated by Corollary \ref{w-cof-prop}. Consider the following exact sequence $$0\rightarrow N^{\p}/N\rightarrow M/N\rightarrow M/N^{\p}\rightarrow 0.$$
Since $(M/N^{\p})_w$ and $(N^{\p}/N)_w$ are $w$-cofinitely generated, $(M/N)_w$ is also $w$-cofinitely generated, which is a contradiction. Therefore $M$ is $w$-Artinian.
\end{proof}

Let $R$ be a ring, $r\in R$ and $M$ a $w$-module over $R$. In the proof of Theorem \ref{main}, we often consider $R$-modules of the form $(rM)_w$ rather than $rM$. Indeed, $(rM)_w\not=rM$, that is, $rM$ is not a $w$-module in general. In fact, the following example shows that  a principal ideal $rR$ need not be a $w$-module even for Noetherian rings $R$, and so the $w$-operation is a star operation rather than a principally star operation in general .

\begin{example}\label{prin-not-w}  Let $D=\mathbb{Q}[x_1,x_2,r,a,b,c,d]$ be a polynomial ring over the field $\mathbb{Q}$ of rational numbers  with $7$ variables. Set $R=\mathbb{Q}[x_1,x_2,r,a,b,c,d]/I$, where $$I=\langle cr-ax_1, cx_2-d-bx_1, rd, ax_2-br, ca-rx_1, r^2-a^2, rx_2-ab\rangle.$$  Let $J=\langle \overline{x_1},\overline{x_2}\rangle$ be an ideal of $R$. One can verify that  $\overline{x_1},\overline{x_2}$ is an $R$-regular sequence by Magama. So the depth of $J$ is 2, and hence $J$ is a $\GV$-ideal of $R$ by \cite[Exercise 6.10]{fk16}. By  Koszul duality $($see \cite[Theorem 1.7]{M74}$)$, we have $$\Ext_R^1(R/J,R\overline{r})\cong \Tor_1^R(R/J,R\overline{r}).$$
Moreover, the latter is isomorphic to $(0:_{R/(R\overline{x_1}+(0:_R\overline{r}))}\overline{x_2}).$ Indeed, since $\overline{x_1},\overline{x_2}$ is an $R$-regular sequence, we have the following quasi-isomorphisms of complexes:
\begin{align*}
&R/J\otimes^{\mathbb{L}}_RR\overline{r}\\
\simeq &R/J\otimes^{\mathbb{L}}_{R/R\overline{x_1}}(R/R\overline{x_1}\otimes^{\mathbb{L}}_RR\overline{r})\\
\simeq &R/J\otimes^{\mathbb{L}}_{R/R\overline{x_1}} R\overline{r}/R\overline{r}\overline{x_1}\\
\simeq & R/J\otimes^{\mathbb{L}}_{R/R\overline{x_1}} R/(R\overline{x_1}+(0:_R\overline{r}))\\
\simeq &[0\rightarrow R/R\overline{x_1}\xrightarrow{\cdot \overline{x_2}}R/R\overline{x_1}\rightarrow 0]\otimes^{\mathbb{L}}_{R/R\overline{x_1}} R/(R\overline{x_1}+(0:_R\overline{r})).
 \end{align*}
Hence, $$\Tor_1^R(R/J,R\overline{r})\cong \Ker([R/(R\overline{x_1}+(0:_R\overline{r}))\xrightarrow{\cdot \overline{x_2}}R/(R\overline{x_1}+(0:_R\overline{r}))])= (0:_{R/(R\overline{x_1}+(0:_R\overline{r}))}\overline{x_2}).$$
Claim that $(0:_{R/(R\overline{x_1}+(0:_R\overline{r}))}\overline{x_2})\not=0$. In fact, $c\not\in R\overline{x_1}+(0:_R\overline{r})$, but $cx_2\in R\overline{x_1}+(0:_R\overline{r})$. So $c$ is a nonzero element in  $(0:_{R/(R\overline{x_1}+(0:_R\overline{r}))}\overline{x_2})$. Consequently, $\Ext_R^1(R/J,R\overline{r})\not=0$. Hence the principal ideal $R\overline{r}$ is not a $w$-ideal of $R$. Note that we verify by the Magma calculation program that $\overline{x_1},\overline{x_2}$ is an $R$-regular sequence, $c\not\in R\overline{x_1}+(0:_R\overline{r})$, and $cx_2\in R\overline{x_1}+(0:_R\overline{r})$ in the final {\rm \textbf{Appendix}}.
\end{example}

Obviously, we can deduce the following corollaries by Theorem \ref{main}.
\begin{corollary} Let $R$ be a ring. A $w$-module $M$ over $R$ is $w$-Artinian if and only if $M$ is $w$-cofinitely generated and  $(M/(0:_M\p))_w$ is $w$-cofinitely generated  for every prime $w$-ideal $\p$ of $R$ with $(0:_RM)\subseteq \p$.
\end{corollary}

\begin{corollary}\cite[Theorem 4.10]{zkh21} Let $R$ be a ring. A $w$-module $M$ over $R$ is $w$-Artinian if and only if $M$ is $w$-cofinitely generated and  $(M/(0:_M\p))_w$ is $w$-cofinitely generated  for every prime $w$-ideal $\p$ of $R$.
\end{corollary}

\begin{corollary} Let $R$ be a ring. A $w$-module $M$ over $R$ is $w$-Artinian if and only if $M$ is $w$-cofinitely generated and  $(M/(M[\p])_w)_w$ is $w$-cofinitely generated  for every prime $w$-ideal $\p$ of $R$ with $(0:_RM)\subseteq \p$, where $M[\p]=\bigcap\limits_{s\in R \setminus \p}s(0:_M\p)$.
\end{corollary}

\begin{corollary} Let $R$ be a ring. A $w$-module $M$ over $R$ is $w$-Artinian if and only if $M$ is $w$-cofinitely generated and  $(M/(M[\p])_w)_w$ is $w$-cofinitely generated  for every prime $w$-ideal $\p$ of $R$, where $M[\p]=\bigcap\limits_{s\in R \setminus \p}s(0:_M\p)$.
\end{corollary}

\appendix{ \textbf{Appendix:}\ The Magma calculation program for Example \ref{prin-not-w}.}
\begin{verbatim}
P<x1,x2,r,a,b,c,d>:=PolynomialRing(RationalField(),7);
I:=ideal<P|c*r-x1*a,x2*c-d-x1*b,d*r,x2*a-b*r,c*a-x1*r,r^2-a^2,x2*r-a*b >;
J:=ideal<P|x1,x2>+I;
K:=ideal<P|x1>+I;
L:=ideal<P|x2>+I;
T:=ideal<P|r>+I;
IdealQuotient(I,K);
A1:=x1*r - a*c in I;
A2:=x2*r - a*b in I;
A3:= r^2 - a^2 in I;
A4:=x1*a - r*c in I;
A5:= x2*a - r*b in I;
A6:= x1*b - x2*c + d in I;
A7:= r*d in I;
A8:= a*d in I;
A1 and A2 and A3 and A4 and A5 and A6 and A7 and A8;
//verify x1,x2 is an R-regular sequence.
IdealQuotient(K,L);
B1:=x2*r - a*b in K;
B2:= r^2 - a^2 in K;
B3:= x2*a - r*b in K;
B4:= x2*c - d in K;
B5:= r*c in K;
B6:= a*c in K;
B7:= r*d in K;
B8:= a*d in K;
B9:=x1 in K;
B1 and B2 and B3 and B4 and B5 and B6 and B7 and B8 and B9;
S1:=IdealQuotient(I,K) eq I;
S2:=IdealQuotient(K,J) eq K;
S3:=c notin IdealQuotient(I,T)+K; //verify  c not in Rx1+(0:r).
S4:=x2*c in IdealQuotient(I,T)+K; //verify  cx2 in Rx1+(0:r).
S1 and S2 and S3 and S4;
\end{verbatim}

\begin{acknowledgement}\quad\\
The author  was supported by the National Natural Science Foundation of China (No. 12061001).
\end{acknowledgement}


\begin{thebibliography}{99}

\bibitem{c50} I. S.  Cohen,  Commutative rings with restricted minimum condition,  Duke Math. J. 17 (1950), 27-42.
\bibitem{N15} J. Elliott,  Rings, Modules, and Closure operations, Springer Monographs in Mathematics. Springer, Cham, 2019.
\bibitem{G72} R. Gilmer, Multiplicative Ideal Theory, Pure and Applied Mathematics, No. 12, Marcel
Dekker, Inc., New York, 1972.
\bibitem{G86} J. S. Golan, Torsion Theories, Pitman Monographs and Surveys in Pure and Applied Mathematics 29, Longman Scientific and Technical, Horlow, 1986.
\bibitem{H82} A. Hiremath, Cofinitely generated and cofinitely related modules, Acta Math. Acad.
Sci. Hungarica 39 (1982) 1-9.
\bibitem{K08} H. Kim, Module-theoretic characterizations of $t$-linkative domains, Comm.
Algebra 36 (2008), 1649-1670.

\bibitem{K35} W. Krull, Idealtheorie. Springer, Berlin, 1935.

\bibitem{M74} E. Matlis, The Koszul complex and duality, Comm. Algebra, 1, (1974), 87-144.
\bibitem{N06} I. Nishitani, A Cohen-type theorem for Artinian modules,   Arch. Math. 87 (2006), 206-210.

\bibitem{pk21} A. Parkash and S. Kour, On Cohen's theorem for modules, Indian J. Pure Appl. Math.  52 (2021), 869-871.

\bibitem{sv72} D. W. Sharpe and P. Vamos,   Injective Modules, Cambridge Univ. Press, Cambridge, 1972.

\bibitem{s94}   P. F.  Smith,  Concerning a theorem of I. S. Cohen,  XIth National Conference of Algebra (Constanta, 1994), An. Stiint. Univ. Ovidius Constanta Ser. Mat. 2 (1994), 160-167.

\bibitem{fq15} F. G. Wang and L. Qiao, The $w$-weak global dimension of commutative rings, Bull. Korean
Math. Soc. 52 (2015), no. 4, 1327-1338.

\bibitem{fk16} F. G. Wang and H. Kim, Foundations of Commutative Rings and Their Modules, Springer, Singapore, 2016.

\bibitem{zkq23} X. L. Zhang, H. Kim, and W. Qi, On two versions of Cohen's theorem for modules, Kyungpook Math. J. 63 (2023), 29-36.

\bibitem{zkq23-1} X. L. Zhang, H. Kim, and W. Qi,  A note on a Cohen-type theorem for Artinian modules, Beitr\"{a}ge zur Algebra und Geometrie, to appear. https://doi.org/10.1007/s13366-022-00671-x

\bibitem{zkh21} D. C. Zhou, H. Kim, and K. Hu, A Cohen-type theorem for $w$-Artinian modules, J. Algebra Appl. 20 (2021), 2150106 (25 pages).

\bibitem{ZKWC}  D. C. Zhou, H. Kim, F. G. Wang, and D. Chen, A new semistar operation on acommutative ring and its applications. Comm. Algebra, 2020, 48 (9):3973-3988.

\bibitem{DF16}  D. C. Zhou and  F. G. Wang,    The direct and inverse limits of $w$-modules,  Comm. Algebra 44(6) (2016), 2495-2500.
\end{thebibliography}
\end{document}